\documentclass{amsart}

\usepackage{amssymb}
\usepackage[all,cmtip]{xy}
\usepackage{tikz-cd}
\usetikzlibrary{matrix}
\usepackage{hyperref}

\newcommand{\injto}{\hookrightarrow}

\newcommand{\longto}{\longrightarrow}

\makeatletter
\newcommand*\rel@kern[1]{\kern#1\dimexpr\macc@kerna}
\newcommand*\widebar[1]{%
  \begingroup
  \def\mathaccent##1##2{%
    \rel@kern{0.8}%
    \overline{\rel@kern{-0.8}\macc@nucleus\rel@kern{0.2}}%
    \rel@kern{-0.2}%
  }%
  \macc@depth\@ne
  \let\math@bgroup\@empty \let\math@egroup\macc@set@skewchar
  \mathsurround\z@ \frozen@everymath{\mathgroup\macc@group\relax}%
  \macc@set@skewchar\relax
  \let\mathaccentV\macc@nested@a
  \macc@nested@a\relax111{#1}%
  \endgroup
}
\makeatother

\DeclareMathOperator\Sym{Sym}
\DeclareMathOperator\Aut{Aut}
\DeclareMathOperator\Diff{Diff}
\DeclareMathOperator\Isom{Isom}
\DeclareMathOperator\Hom{Hom}

\DeclareMathOperator\Spec{Spec}

\DeclareMathOperator\Pic{Pic}

\DeclareMathOperator\GL{GL}
\DeclareMathOperator\SL{SL}
\DeclareMathOperator\Gal{Gal}
\DeclareMathOperator\Sch{Sch}

\DeclareMathOperator\Mod{Mod}
\DeclareMathOperator\Ab{Ab}

\DeclareMathOperator\sgn{sgn}

\def\bQ{{\mathbf{Q}}} \def\bZ{{\mathbf{Z}}} 
\def\bF{{\mathbf{F}}}  \def\bR{{\mathbf{R}}}
\def\bC{{\mathbf{C}}}  \def\bP{{\mathbf{P}}}

\def\cO{{\mathcal{O}}}  \def\cC{{\mathcal{C}}}
 \def\cE{{\mathcal{E}}} \def\cL{{\mathcal{L}}} 
\def\cM{\mathcal M} \def\cI{\mathcal I} 
\def\cH{\mathcal H} \def\cD{{\mathcal D}} \def\cQ{{\mathcal{Q}}}
\def\cA{{\mathcal A}}

\def\rB{{\mathrm B}} \def\rR{{\rm R}} \def\rH{{\rm H}} 
\def\rL{{\mathrm L}}

\def\id{{\rm id}}

\def\et{{\rm et}}
\def\an{{\rm an}}

\def\tors{{\rm tors}}

\theoremstyle{plain}
\newtheorem{theorem}{Theorem}
\newtheorem{corollary}{Corollary}
\newtheorem{lemma}{Lemma}

\newtheorem{proposition}{Proposition}
\theoremstyle{definition}
\newtheorem{definition}{Definition}
\newtheorem{example}{Example}

\newtheorem{remark}{Remark}

\begin{document}

\title
 {Characteristic classes for curves of genus one}

\author
 {Lenny Taelman}

\maketitle

\begin{abstract}We compute the cohomology of the stack $\cM_1$ over $\bC$ with coefficients in $\bZ[\frac{1}{2}]$, and in low degrees with  coefficients in $\bZ$. Cohomology classes on $\cM_1$ give rise to \emph{characteristic classes}, cohomological invariants of families of curves of genus one. We prove a number of vanishing results for  those characteristic classes, and give  explicit examples of families with non-vanishing characteristic classes. 
\end{abstract}

\section{Introduction and statement of the results}

\subsection{The cohomology of $\cM_1$}
We denote by $\cM_{1}$  the algebraic stack of curves of genus one and by $\cM_{1,1}$ the algebraic stack of elliptic curves, both over $\bC$. See \S \ref{sec:curves} for more details. If $X$ is an algebraic stack of finite type over $\bC$ then we denote by $X^\an$ its analytification and by  $\rH^\bullet(X^\an,-)$ its singular cohomology.

Consider the map $J\colon \cM_{1}\to\cM_{1,1}$, sending a curve to its Jacobian, and its Leray spectral sequence
\begin{equation}\label{eq:intro-leray}
	E_2^{p,q} = \rH^p(\cM_{1,1}^\an, \,\rR^q J_\ast \bZ ) \Longrightarrow \rH^{p+q}(\cM_1^\an,\bZ).
\end{equation}
The fibers of $J$ are classifying spaces of rank two tori, and one has
\[
	\rR^q J_\ast \bZ = \begin{cases}
		\Sym^k \rR^1 \pi_\ast \bZ & (q=2k) \\
		0 & (\text{$q$ odd}),
	\end{cases}
\]
where $\pi\colon \cE \to \cM_{1,1}$ is the universal elliptic curve. The upper-half plane is contractible, and since it is a universal covering of $\cM_{1,1}^\an$ with covering group $\SL_2 \bZ$, the cohomology of local systems on $\cM_{1,1}$ can be expressed in terms of group cohomology for $\SL_2 \bZ$. One finds:
\[
	E_2^{p,q} = \begin{cases}
		\rH^p(\SL_2 \bZ,\, \Sym^k (\bZ^2)) & (q=2k) \\
		0 & (\text{$q$ odd}),
	\end{cases}
\]
where $\bZ^2$ is the standard representation of $\SL_2 \bZ$.

The group $\SL_2 \bZ$ has a free subgroup of index $12$, so if $M$ is an $\SL_2 \bZ$-module on which $6$ is invertible, then $\rH^\bullet(\SL_2 \bZ, M)$ is concentrated in degrees $0$ and $1$. It immediately follows that the spectral sequence (\ref{eq:intro-leray}) tensored with $\bZ[\frac{1}{6}]$ degenerates at $E_2$. We show that it degenerates already with $\bZ[\frac{1}{2}]$-coefficients, and moreover, that the filtration on cohomology splits. In other words:

\begin{theorem} \label{thm:one}
$\rH^n( \cM_1^\an,\, \bZ[\frac{1}{2}] ) = \bigoplus_{p+2k=n} \rH^p( \SL_2 \bZ,\, \Sym^k (\bZ[\frac{1}{2}]^2) )$.
\end{theorem}

We actually prove a more general result about the cohomology of bundles of classifying spaces of tori of arbitrary dimension. The proof is of simplicial nature, and relies on a theorem of Quillen on derived exterior powers. See section \ref{sec:splitting}, and Theorem \ref{thm:torus-split} in particular.

Together with the recent computation of $\rH^\bullet( \SL_2 \bZ,\, \Sym^k(\bZ^2) )$ by Callegaro, Cohen and Salvetti \cite{CallegaroCohenSalvetti13}, Theorem \ref{thm:one} gives a complete description of the cohomology of $\cM_1^\an$ with coefficients in $\bZ[\frac{1}{2}]$.

For small $n$ one can completely compute $\rH^n(\cM_1^\an,\, \bZ)$. Note that the forgetful map $\cM_{1,1}\to \cM_{1}$ is a section to $J$ so that for any group $A$ the map $J^\ast \colon \rH^\bullet(\cM_{1,1}^\an,A) \to \rH^\bullet(\cM_{1}^\an,A)$ is split injective. We denote the resulting complement of the image by $\rH^\bullet(\cM_{1}^\an,A)^\dagger$.

\begin{theorem}\label{thm:low-degree}
We have
$ \rH^\bullet(\cM_{1}^\an,\bZ) = 
	\rH^\bullet(\cM_{1,1}^\an,\bZ) \oplus \rH^\bullet(\cM_{1}^\an,\bZ)^\dagger$
and for $n<6$ these groups are as follows:
\[
 \renewcommand{\arraystretch}{1.4}
\begin{tabular}{ c || c | c | c | c | c | c   } % | c | c }
$n$ & $0$ & $1$ & $2$ & $3$ & $4$ & $5$  \\ % $8$ & $9$ \\
  \hline                        
$\rH^n(\cM_{1,1}^\an,\bZ)$ &
	$\bZ$ &
	$0$ & 
	$\bZ/12$ &
	$0$ & 
	$\bZ/12$ &
	$0$ \\
$\rH^n(\cM_1^\an,\bZ)^\dagger$ & 
	$0$ & 
	$0$ & 
	$0$ &
	$0$ &
	$\bZ/2$ &
	$\bZ/2 \oplus \bZ$ 
\end{tabular}
\]
\end{theorem}

See Theorem \ref{thm:low-degree-repeated} in section \ref{sec:low-degree} for a more extensive table for $n\leq 9$.
This theorem is shown by analyzing the spectral sequence (\ref{eq:intro-leray}), and comparing it with the computations of \cite{CallegaroCohenSalvetti13} and \cite{Furusawa88}. It turns out that in low degree the spectral sequence (\ref{eq:intro-leray}) degenerates at $E_2$, but it is not clear if this happens in arbitrary degree.

\begin{remark}
Since the table in Theorem \ref{thm:low-degree} may lead the reader into suspecting that $\rH^\bullet(\cM_{1}^\an,\,\bZ)$ contains only $2$ and $3$-power torsion, we briefly point out that in fact it contains $p$-torsion for all primes $p$. Indeed, if $p$ is a prime then
\[
	X^pY-Y^pX\in \bF_p[X,Y]
\]
is invariant under $\SL_2 \bF\!{}_p$, so 
$\rH^0( \SL_2 \bZ,\, \Sym^{p+1} (\bF_{\!p}^2) )$ is non-zero. It follows that 
\[
	\rH^1(\SL_2 \bZ,\, \Sym^{p+1} (\bZ^2) ),
\]
contains $p$-torsion, and by Theorem \ref{thm:one} we conclude that for any prime $p>2$ the group $\rH^{2p+3}( \cM_1^\an, \bZ)$ has $p$-torsion. See \cite{Furusawa88} for a complete description of the $p$-power torsion of $\rH^\bullet(\cM_{1}^\an,\,\bZ)$ for $p\geq 5$.
\end{remark}

\begin{remark}By Shimura \cite{Shimura59}  cohomology of $\SL_2 \bZ$  is related to modular forms of level $1$, and indeed, it was already observed by Morita \cite{Furusawa88} that the spectral sequence (\ref{eq:intro-leray})  implies
\[
	\rH^n( \cM_1^\an,\, \bC ) = \begin{cases}
		\bC & (n=0) \\
		E_k \oplus S_k \oplus \bar S_k & (\text{$n=2k-3$})\\
		0 & (\text{$n>0$ even}), \end{cases}
\]
where $S_k$ is the complex vector space of level $1$ and weight $k$ cusp forms, $\bar S_k$ its complex conjugate vector space and
$E_k$ the space (of dimension zero or one) generated by the Eisenstein series of weight $k$. 
See \cite[\S 4.3]{Behrend03} for an $\ell$-adic version. 
\end{remark}

\begin{remark}\label{rmk:teichmuller} Earle and Eells \cite[\S 10]{EarleEells69} have shown that the space of complex structures on an oriented differentiable $2$-torus $T$ that are compatible with the given orientation is contractible. It follows that (the geometric realization of the topological category) $\cM_{1}^\an$ is homotopy equivalent with the classifying space of the group $\Diff_+ T$ of orientation-preserving self-diffeomorphisms of $T$ and therefore $
\rH^\bullet(\cM_{1}^\an, \bZ) = \rH^\bullet(\rB \Diff_+ T, \bZ)$.
\end{remark}

\subsection{Characteristic classes}

Let $A$ be an abelian group. Let $\gamma$ be a class in $\rH^\bullet(\cM_{1}^\an,A)$. A curve of genus one $C\to S$ corresponds to a map $f\colon S\to\cM_{1}$. If $S$ is of finite type over $\bC$ then we set
\[
	c_\gamma(C/S) := f^\ast \gamma \in \rH^\bullet(S^\an,A).
\]
In this way $\gamma$ defines a \emph{characteristic class} $c_\gamma$. By construction, we have that
$c_\gamma$ commutes with base change $S'\to S$. If moreover $\gamma$ lies in the dagger part $\rH^\bullet(\cM_1^\an,A)^\dagger$ then we also have $c_\gamma(C/S)=0$ if $C\to S$ has a section.

In the final sections, we study these characteristic classes $c_\gamma$ a bit more in detail.
For  characteristic zero coefficients we prove the following.

\begin{theorem}\label{thm:smooth-vanishing}
Let $\gamma \in \rH^\bullet(\cM_1^\an,\,\bQ)^\dagger$. 
\begin{enumerate}
\item If $S$ is a smooth scheme of finite type over $\bC$ and $C\to S$ is a curve of genus one then $c_\gamma(C/S)=0$;
\item If $\gamma \neq 0$ then there exists a scheme $S$ of finite type over $\bC$ and $C\to S$ such that $c_\gamma(C/S)\neq 0$.
\end{enumerate}
\end{theorem}
For (2) we give an explicit construction of a (necessarily singular) $S$ and a family $C\to S$ with $c_\gamma(C/S) \neq 0$. Of course, if we allow $S$ to be an algebraic stack, then (2) becomes tautological since one can take $C\to S$ to be the universal curve $\cC\to \cM_1$.

\begin{remark}
The definition of the characteristic classes $c_\gamma(C/S)$ extends to families of genus one Riemann surfaces over complex analytic manifolds. If $S$ is a smooth, finite type scheme over $\bC$, and $C\to S^\an$ an analytic family of genus one Riemann surfaces, then by the first part of Theorem \ref{thm:smooth-vanishing} the characteristic classes $c_\gamma(C/S^\an)$  for $\gamma \in \rH^\bullet(\cM_1^\an,\,\bQ)^\dagger$ are obstructions against the algebraizability of $C\to S^\an$. 
\end{remark}

 The element of order $2$ in $\rH^4(\cM_1^\an,\bZ)^\dagger$ (see Theorem \ref{thm:low-degree}) determines a non-zero $\gamma \in \rH^3(\cM_1^\an,\bZ/2\bZ)$. In section \ref{sec:explicit-torsion-classes} we study the associated characteristic class in more detail. 
 
We construct a class
 \[
 	c(C/S) \in \rH^3(S_\et,\,\bZ/2\bZ)
\]
for every curve of genus one $C$ over a scheme $S$ over $\Spec \bZ[\frac{1}{2}]$, directly in terms of the geometry of $C\to S$, and show that $c=c_\gamma$ for $S$ of finite type over $\bC$. 

The class $c(C/S) \in \rH^3(S_\et,\,\bZ/2\bZ)$ is compatible with base change, and vanishes if $C\to S$ has a section. Being a torsion class, it need not vanish for curves over a smooth base $S$, and indeed we give an example of a smooth scheme $S$ of finite type over $\bC$ (of dimension $3$) and a genus one curve $C\to S$ such that $c(C/S)\neq 0$.  We end by showing that $c(C/S)=0$ if $S$ is the spectrum of a number field.

\subsection*{Acknowledgements} The author is grateful to Bhargav Bhatt, Niels uit de Bos, Bas Edixhoven,  H\'el\`ene Esnault, Bart de Smit and Wouter Zomervrucht for various discussions related to this paper.

The author is supported by a grant of the Netherlands Organisation for Scientific Research (NWO).

\section{Genus one curves and their jacobians}\label{sec:curves}

We recall a few basic results on curves of genus one and on the stacks $\cM_1$ and $\cM_{1,1}$, mostly for lack of proper reference. We also use this as an opportunity to fix  notation and terminology.

For the language of algebraic spaces and stacks, we use the conventions of the \emph{Stacks Project}. See in particular \cite[\href{http://stacks.math.columbia.edu/tag/025R}{Tag 025R}]{stacks-project} and \cite[\href{http://stacks.math.columbia.edu/tag/026N}{Tag 026N}]{stacks-project}.

\begin{definition} Let $S$ be a scheme. A \emph{curve of genus one} over $S$ is a proper smooth morphism of algebraic spaces $p\colon C \to S$,  whose geometric fibers are schemes which are irreducible curves of genus one.
\end{definition}

\begin{definition} Let $S$ be a scheme. An \emph{elliptic curve} over $S$ is a pair $(E,0)$ consisting of a proper smooth scheme $\pi \colon E\to S$ that is a curve of genus one, and a section $0\in E(S)$.
\end{definition}

Of course, an elliptic curve is canonically a group scheme over $S$.

\begin{lemma}\label{lemma:etale-locally-section}
 Let $C\to S$ be a curve of genus one. Then there exists an \'etale surjective $S' \to S$ such that the base change $C'\to S'$ has a section.
\end{lemma}

\begin{proof} Since $C$ is an algebraic space, there exists a scheme $U$ and an \'etale surjective $U\to C$. The composition $U\to S$ is smooth and surjective, so that by \cite[17.6.3]{EGA-IV-4} there exists an \'etale surjective $S'\to S$ and a section $S'\to U':=S'\times_S U$. Composition with $U'\to C'$ gives the desired section.
\end{proof}

\begin{lemma}\label{lemma:finite-presentation}
Let $S$ be an affine scheme and let $C\to S$ be a curve of genus one. Then there exists a 
cartesian square
\[
\begin{tikzcd}
C \arrow{r} \arrow{d} & C_0 \arrow{d} \\
S \arrow{r} & S_0
\end{tikzcd}
\]
with $S_0$ affine of finite type over $\Spec \bZ$ and $C_0\to S_0$ a curve of genus one. If moreover $s:S\to C$ is a section, then $C_0\to S_0$ can be taken such that there is a section $s_0\colon S_0\to C_0$ inducing $s$.
\end{lemma}

\begin{proof}
Since $C$ is quasi-compact,  there is an \'etale surjective $U\to C$ with $U$ a quasi-compact scheme. In particular $U$ is of finite presentation over $S$. Similarly, $R := U\times_C U$ is of finite presentation over $U$ and over $S$. We find that $C$ is the quotient of an equivalence relation $\left[ R \rightrightarrows U \right]$ with $U$, $R$ and the arrows of finite presentation. All these data can be defined over a finitely generated subring of $\Gamma(S,\cO_S)$ and the first claim follows.

For the second claim, take $\left[ R \rightrightarrows U \right]$ as above. Put $S':=U\times_C S$, using the section $S\to C$, and $S'' := S' \times_S S'$. Then $S'$ and $S''$ are \'etale schemes of finite presentation, and the section $S\to C$ induces a map 
\[
	\left[ S'' \rightrightarrows S' \right] \to \left[ R \rightrightarrows U \right]
\]
which determines $s$. Again we see that all data can be defined over a finitely generated subring of $\Gamma(S,\cO_S)$.
\end{proof}

\begin{lemma}\label{lemma:section-implies-representable}
A curve of genus one with a section is representable by an elliptic curve.
\end{lemma}

\begin{proof}
Let $p\colon C\to S$ be a curve of genus one, and let $s\colon S\to C$ be a section. By Lemma \ref{lemma:finite-presentation} we may assume that $S$ is affine and noetherian. Then the image of $s$ is closed in $C$. Let $\cI$ be the ideal sheaf corresponding to the image.

We claim that $\cI$ is an invertible sheaf on $C$. It suffices to check this \'etale locally. Choose an \'etale surjective  $U\to C$ of finite presentation, with $U$ a scheme. Base-changing everything to $S':=U\times_C S$, we may  assume that the section $S\to C$ factors over $S\to U$. Now $U$ is a representable smooth curve over $S$ with a section, and the pull-back of $\cI$ to $U$ is the ideal sheaf of this section, hence invertible. 

Let $\cL := \cI^{\otimes -3}$. We claim that $p_\ast \cL$ is locally free of rank $3$. Indeed, for an $s\in S$ denote by $p_s$ the fiber $X_s \to \Spec k(s)$. Then for every $s\in S$ and $i>0$ we have $\rR^i p_{s,\ast} \cL_s=0$. It follows that $\rR^i p_\ast \cL=0$ for all $i>0$ and that the formation of $p_\ast \cL$ commutes with base change. Since for every $s$ the vector space $p_{s,\ast}\cL_s$ is of dimension $3$ over $k(s)$, we conclude that $p_\ast \cL$ is locally free of rank $3$.

Shrinking $S$ we may assume that $p_\ast \cL$ is free of rank $3$, and choosing a basis we obtain a morphism $i\colon C\to \bP^2_S$ over $S$. Its fibers $i_s\colon C_s \to \bP^2_{k(s)}$ are closed immersions. In particular, if $U\to C$ is \'etale surjective and $U$ quasi-compact then the composition $U\to \bP^2_S$ is quasi-finite. We conclude that $C\to \bP^2_S$ is quasi-finite, and hence with \cite[II.6.16]{Knutson71} that $C$ is a scheme.
\end{proof}

\begin{definition}[relative Picard functor] Let $C\to S$ be a curve of genus one and $n$ an integer. Then we define
$\Pic^n_{C/S}$ to be the sheafification of the presheaf
\[
	T \,\mapsto\, \{ \cL \in \Pic C_T \mid \deg \cL_s = n \text{ for all $s\in S$} \} / \Pic T
\]
on $(\Sch/S)_\et$. 
\end{definition}

\begin{lemma} Let $C \to S$ be a curve of genus one. Then $\Pic^1_{C/S}=C$ and $\Pic^0_{C/S}$ is an elliptic curve over $S$. Moreover, $C$ is a torsor under $\Pic^0_{C/S}$ on $(\Sch/S)_\et$.
\end{lemma}
 
\begin{proof}
Given a section $s\in C(T)$ we have a line bundle $\cL(s) \in \Pic^1_{C/S}(T)$. This defines a map of sheaves 
$C\to \Pic^1_{C/S}$. By Lemmas \ref{lemma:etale-locally-section} and \ref{lemma:section-implies-representable} the algebraic space $C$ is \'etale locally on $S$ representable by an elliptic curve, hence the map $C\to \Pic^1_{C/S,\et}$ is \'etale locally an isomorphism, and therefore an isomorphism.

Similarly, the sheaf of groups $\Pic^0_{C/S}$ is \'etale locally representable by an elliptic curve, and therefore, by descent, it is an elliptic curve over $S$.  By Lemma \ref{lemma:etale-locally-section} the sheaf $\Pic^1_{C/S}$ has \'etale locally a section, so it is a torsor under $\Pic^0_{C/S}$.
\end{proof}

We conclude

\begin{proposition}\label{prop:curves-versus-torsors}
Let $S$ be a scheme. The functor
\[
	C \mapsto (\Pic^0_{C/S}, \Pic^1_{C/S} )
\]
defines an equivalence of groupoids between
\begin{enumerate}
\item curves of genus one over $S$, and,
\item pairs of an elliptic curve $E/S$ and an $E$-torsor on $(\Sch/S)_\et$.
\end{enumerate}
This equivalence is compatible with base change along maps $S'\to S$.\qed
\end{proposition}

We denote by $\cM_1(S)$ and $\cM_{1,1}(S)$ the groupoids of curves of genus one
over $S$, and of elliptic curves over $S$ respectively. Varying $S$ we obtain categories $\cM_1$ and $\cM_{1,1}$ fibered in groupoids over $\Sch$. We have a 
\emph{Jacobian map} $J\colon \cM_1 \to \cM_{1,1}$, sending $C$ to $\Pic^0_{C/S}$.
The \emph{forgetful map} $f\colon \cM_{1,1} \to \cM_1$, sending $(E,0)$ to $E$, is a section of $J$.

\begin{theorem} \label{thm:stacks-summary}
$\cM_1$ and $\cM_{1,1}$ are algebraic stacks over $\Spec \bZ$. Moreover,
\begin{enumerate}
\item $\cM_{1,1}$ is Deligne-Mumford, separated and smooth over $\Spec \bZ$;
\item $\cM_1$ is separated and smooth over $\Spec \bZ$;
\item $f\colon \cM_{1,1} \to \cM_1$ is representable by algebraic spaces, proper, smooth, and coincides with the universal curve $\cC \to \cM_1$;
\item $J\colon \cM_1 \to \cM_{1,1}$ coincides with the classifying space $B\cE \to \cM_{1,1}$ of the universal elliptic curve $\cE \to \cM_{1,1}$. 
\end{enumerate}
\end{theorem}

\begin{proof}
For (3), let $S$ be a scheme and $S\to \cM_1$ a map corresponding to a curve $C\to S$ of genus one. Then $S\times_{\cM_1} \cM_{1,1}$ is the sheaf of sections of $C$, and hence naturally isomorphic with $C$. This shows that $\cM_1 \to \cM_{1,1}$ is representable by an algebraic space, and that it coincides with the universal curve $\cC \to \cM_1$. In particular, it is proper and smooth.

Now for (1) and (2), we have that 
$\cM_{1,1}$ is a separated smooth Deligne-Mumford stack by \cite{DeligneRapoport73}. $\cM_1$ is a stack by the description in terms of torsors under elliptic curves. If $C_1$ and $C_2$ are genus one curves over $S$ 
with Jacobians $E_1$ and $E_2$ respectively, then $\Isom(C_1,C_2) \to S$ factors
as
\[
	\Isom(C_1,C_2) \to \Isom(E_1,E_2)\to S.
\]
The map $\Isom(C_1,C_2) \to \Isom(E_1,E_2)$ is an $E$-torsor (for $E$ the base change of either $E_1$ or $E_2$ to $\Isom(E_1,E_2)$), so is representable by a proper smooth algebraic space. The map $\Isom(E_1,E_2) \to S$ is finite, and we conclude that $\Isom(C_1,C_2)\to S$ is proper. It follows that the diagonal of $\cM_1$ is representable and proper.  Any choice of a scheme $U$ and smooth surjective $U \to \cM_{1,1}$ induces a smooth surjective $U\to \cM_1$.
(For example one may take $U=Y(n)_{\bZ[1/n]} \coprod Y(m)_{\bZ[1/m]}$ for coprime integers $n,m$ with $n,m \geq 3$.) Since $U\to \Spec \bZ$ is smooth, the algebraic stack $\cM_1$ is smooth over $\Spec \bZ$.

Finally, (4) is a restatement of Proposition \ref{prop:curves-versus-torsors}.
\end{proof}

\begin{theorem}[Raynaud {\cite[XIII.2.6]{Raynaud70}}]\label{thm:Raynaud}
 Let $S$ be a quasi-compact scheme and $C\to S$ a curve of genus one, and $E\to S$ its relative Jacobian.
\begin{enumerate}
\item If $S$ is normal then $C\to S$ is representable by a scheme if and only if $C$ is of finite order as an $E$-torsor;
\item If $S$ is regular then $C\to S$ is of finite order as an $E$-torsor and representable by a scheme.
\end{enumerate}
\end{theorem}

See \cite[XIII 3.2]{Raynaud70} or \cite{Zomervrucht13} for an example of a $C\to S$ which is not representable by a scheme. We end this section with an example of a representable genus one curve of infinite order. We will use a similar construction in \S \ref{sec:Q-coefficients}.

\begin{example}[{\cite[XIII 3.1]{Raynaud70}}] 
\label{ex:infinite-order}
Let $E$ be an elliptic curve over a field $k$ and let $x\in E(k)$. Let $S$ be the nodal curve $S$ obtained by identifying $\{0\}$ and $\{\infty\}$ in $\bP^1$. Let $p\colon C\to S$ be the curve obtained by identifying the closed subschemes $E\times \{0\}$ and $E\times \{\infty\}$ of $E\times \bP^1$ by translation along $x\in E(k)$, in other words, by identifying points $(a,0)$ with $(a+x,\infty)$ in $E\times \bP^1$. The scheme $C\to S$ exists by \cite[Theorem 5.4]{Ferrand03}. To see that $C\to S$ is a genus one curve it suffices to observe that the pull-back to the \'etale neighbourhood $U = \Spec k[x,y]/(xy) \to S$ of the node $s \in S$ is isomorphic to $E_U \to U$. 

The torsor $C$ is trivial if and only if it has a section, and since there are no non-constant maps $\bP^1\to E$, this happens only if $x=0$. In particular, one finds that the $E_S$-torsor $C\to S$ is of infinite order as soon as $x\in E(k)$ is of infinite order.
\end{example}

\section{Cohomology with $\bZ[\frac{1}{2}]$-coefficients}\label{sec:splitting}

\subsection{A splitting result in group cohomology}

\begin{proposition}\label{prop:equivariant-group-cohomology}
Let $\Lambda$ be a free abelian group of finite rank $d$. Let $G$ be a subgroup of $\Aut \Lambda$. Consider the left-exact functor
\[
	\pi_\ast \colon \Mod_{\bZ[G\ltimes \Lambda]} \to \Mod_{\bZ[G]},\, M \mapsto M^\Lambda.
\]
Then
\[
	\rR \pi_\ast \bZ[\textstyle\frac{1}{d!}] 
	\cong \bigoplus_k \left( \wedge^k\Hom(\Lambda,\bZ[\frac{1}{d!}])\right)[-k]
\]
in $\cD^+(\Mod_{\bZ[G]})$.
\end{proposition}

We have  $\rR^k \pi_\ast \bZ \cong \wedge^k \Hom(\Lambda, \bZ)$ as $\bZ[G]$-modules for all $k$. However,
the complex $\rR \pi_\ast \bZ \in \cD^+(\Mod_{\bZ[G]})$ is in general not split, see \cite{LangerLueck12} for a counterexample.

\begin{proof}[Proof of Proposition \ref{prop:equivariant-group-cohomology}]
The bar-resolution of the $\bZ[\Lambda]$-module $\bZ$ is in fact a resolution of $\bZ[G\ltimes \Lambda]$-modules that is acyclic for $\pi_\ast$, so it can be used to compute $\rR \pi_\ast \bZ$. The result is the usual cochain complex $C^\bullet(\Lambda,\bZ)$ with its natural action of $G$. In detail, for $n\geq 0$ let $C^n(\Lambda,\bZ)$ be the set of maps from $\Lambda^n$ to $\bZ$, and consider the map
\[
	d^n\colon C^n(\Lambda,\bZ) \to C^{n+1}(\Lambda,\bZ)
\]
given by
\begin{eqnarray*}
	(d^n\!f)(\lambda_1,\ldots, \lambda_{n+1}) &=& f(\lambda_2,\ldots,\lambda_{n+1}) \\
	&+& \sum_{i=1}^n (-1)^i f(\lambda_1,\ldots, \lambda_i+\lambda_{i+1}, \ldots, \lambda_{n+1}) \\
	&+& (-1)^{n+1}f(\lambda_1,\ldots, \lambda_n).
\end{eqnarray*}
Then $C^\bullet(\Lambda,\bZ)$ is a complex of $\bZ[G]$-modules, quasi-isomorphic to $\rR \pi_\ast \bZ$. Similarly we have $C^\bullet(\Lambda,\bZ[\frac{1}{d!}]) \cong
\rR \pi_\ast \bZ[\frac{1}{d!}]$.

Now consider for $0\leq k \leq d$ the maps
\[
	a^k\colon \wedge^k \Hom(\Lambda,\,\bZ[\textstyle\frac{1}{d!}])
	\to C^k(\Lambda,\bZ[\textstyle\frac{1}{d!}] )
\]
given by
\[
	\varphi_1\wedge\cdots\wedge\varphi_k \mapsto
	\Big[ (\lambda_1,\ldots,\lambda_k) \mapsto 
	\frac{1}{k!} \sum_{\sigma \in S_k}
		\sgn(\sigma) \varphi_1(\lambda_{\sigma 1}) \cdots \varphi_k(\lambda_{\sigma k})
	\Big].
\]
The collection of maps $a^k$ define a quasi-isomorphism
\[
	a\colon \bigoplus_k
	\big( \wedge^k \Hom(\Lambda,\,\bZ[\textstyle\frac{1}{d!}]) \big)[-k]
	 \longrightarrow C^\bullet(\Lambda,\bZ[\textstyle\frac{1}{d!}])
\]
of complexes of $\bZ[G]$-modules. Indeed, the maps $a^k$ are clearly $G$-equivariant, and a direct computation shows that $d^ka^k(\varphi_1\wedge\cdots\wedge\varphi_k)=0$, so that $a$ is indeed a morphism of complexes. 

Now both source and target of $a$ are differential graded algebras whose cohomology ring is the exterior power ring on $\rH^1=\Hom(\Lambda,\bZ[\frac{1}{d!}])$. Also, $a$ induces the identity map on $\rH^1$. So to show that $a$ is a quasi-isomorphism, it suffices to show that $a$ respects the product structure on the cohomology rings.

Indeed, let $\varphi_1,\ldots,\varphi_k \in  \Hom(\Lambda,\bZ[\frac{1}{d!}])$. Then the product of the cochains $a\varphi_i$ is the map
\[
	\Lambda^k \to \bZ[\textstyle\frac{1}{d!}],\, (\lambda_1,\ldots,\lambda_k) \mapsto \varphi_1(\lambda_1)  \cdots \varphi_k(\lambda_k).
\]
The product on cohomology being graded commutative, we see that this gives the same class in
$\rH^k(C^\bullet(\Lambda,\bZ[\frac{1}{d!}]))$ as the cochain
\[
	(\lambda_1,\ldots,\lambda_k)  \mapsto \frac{1}{k!}
	\sum_{\sigma \in S_k} \sgn(\sigma) \varphi_1(\lambda_{\sigma i}) \cdots \varphi_k(\lambda_{\sigma k}),
\]
which is precisely the image of $\varphi_1\wedge\cdots \wedge \varphi_k$ under $a^k$. We conclude that $a$ respects the product structures and therefore induces an isomorphism on cohomology.
\end{proof}

\subsection{Cohomology of bundles of classifying spaces of tori}
Let $d$ be a non-negative integer. Let $S$ be a topological space and let $\Lambda$ be a sheaf of abelian groups on $S$, locally free of rank $d$. Then $\Lambda \otimes_\bZ S^1$ defines a relative torus $\pi\colon T\to S$ with $\rR^1\pi_\ast \bZ = \Hom(\Lambda,\bZ)$. 

\begin{theorem}\label{thm:torus-split}Let $\Lambda$ be a free abelian group of finite rank $d$.  Let $G$ be a subgroup of $\Aut \Lambda$. Let $\pi \colon T \to \rB G$ be the relative torus  with $\rR^1\pi_\ast \bZ = \Hom(\Lambda,\bZ)$. Let $g\colon \rB(T/\rB G)\to \rB G$ be its relative classifying space. Let 
\[
	\cH := \Hom(\Lambda,\,\bZ[\textstyle\frac{1}{d!}])=\rR^1\pi_\ast \bZ[\textstyle\frac{1}{d!}],
\]
in $\Ab(\rB G)$. Then we have isomorphisms
\begin{enumerate}
\item $\rR \pi_\ast \bZ[\frac{1}{d!}] \cong \bigoplus_k ( \wedge^k \cH)[-k]$, 
\item $\rR g_\ast \bZ[\frac{1}{d!}] \cong \bigoplus_k ( \Sym^k \cH)[-2k]$
\end{enumerate}
in $\cD^+(\Ab(\rB G))$.
\end{theorem}

By Borel \cite[\S19]{Borel53} one has
\[
	\rR^q g_\ast \bZ = \begin{cases}
		\Sym^k \cH & (q=2k) \\
		0 & (\text{$q$ odd})
	\end{cases}
\]
but it seems unlikely that in general $\rR g_\ast \bZ$ splits as a complex of $\bZ[G]$-modules.

\begin{proof}[Proof of Theorem \ref{thm:torus-split}]
Assertion (1) is just a restatement of Proposition \ref{prop:equivariant-group-cohomology}, since $T=\rB(G\ltimes \Lambda)$ and $\pi$ is the map induced by $G\ltimes \Lambda \to G$. 

We will prove (2) using a simplicial computation. For a non-negative integer $i$, let $T_i$ be the $i$-fold fiber product $T \times_{\rB G} \ldots \times_{\rB G} T$ over $\rB G$, with structure map $\pi_i \colon T_i \to \rB G$.  Then $\rB(T/\rB G)$ can be represented by a simplicial topological group, relative over $\rB G$:
\[
\xymatrix{
\cdots 
\ar@<1.5ex>[r]
\ar@<.5ex>[r]
\ar@<-.5ex>[r]
\ar@<-1.5ex>[r]
&
T_2
\ar@<1ex>[r]
\ar@<0ex>[r]
\ar@<-1ex>[r]
&
T_1
\ar@<.5ex>[r]
\ar@<-.5ex>[r]
&
T_0
}
\]
(here we have only drawn the degeneracy maps). If we have complexes $C_i^\bullet$ representing $\rR \pi_{i,\ast} \bZ[\frac{1}{d!}]$ and maps $C_i^\bullet \to C_{i+1}^\bullet$ representing pull-backs along the various degeneracy maps, then the total  complex of the double complex associated to
\[
\xymatrix{
	C_0^\bullet \ar@<.5ex>[r] \ar@<-.5ex>[r]
	& C_1^\bullet \ar@<1ex>[r] \ar@<0ex>[r] \ar@<-1ex>[r]
	& C_2^\bullet \ar@<1.5ex>[r]\ar@<.5ex>[r]\ar@<-.5ex>[r]\ar@<-1.5ex>[r] & \cdots
}
\]
will be quasi-isomorphic to $\rR g_\ast \bZ[\frac{1}{d!}]$.

Using (1) and K\"unneth we have
\[
	\rR \pi_{i,\ast} \bZ[\textstyle\frac{1}{d!}] = \bigoplus_k (\wedge^k (\cH^i))[-k].
\]
Pull-back along the $j$-th degeneracy map $d_j^i: T_{i+1} \to T_i$ induces a map
\[
	\rR \pi_{i,\ast} \bZ[\textstyle\frac{1}{d!}] \longto \rR \pi_{i+1,\ast} \bZ[\frac{1}{d!}].
\]
In degree $k=1$ it is given by
\[
	\cH^i \to \cH^{i+1},\,
	(s_1,\ldots,s_i) \mapsto 
	\begin{cases}
		(0,s_1,\ldots,s_i) & (j=0) \\
		(s_1,\ldots, s_j,s_j,\ldots, s_i) & (1\leq j \leq i ) \\
		(s_1,\ldots,s_i,0) & (j=i+1)
	\end{cases}
\]
and in degree $k$ by applying the functor $\wedge^k$ to the above map. In other words, we have co-simplicial $\bZ[G]$-modules
\[
A_k := \Big[
\xymatrix{
	\wedge^k0 \ar@<.5ex>[r] \ar@<-.5ex>[r]
	& \wedge^k \cH \ar@<1ex>[r] \ar@<0ex>[r] \ar@<-1ex>[r]
	& \wedge^k (\cH^2) \ar@<1.5ex>[r]\ar@<.5ex>[r]\ar@<-.5ex>[r]\ar@<-1.5ex>[r] & \cdots
}\Big]
\]
with associated cochain complexes $C^\bullet(A_k)$, and a quasi-isomorphism
\begin{equation}\label{eq:double-simplicial-decomposed}
	\rR g_\ast \bZ[\textstyle\frac{1}{d!}] \cong \bigoplus_k C^\bullet(A_k)[-k].
\end{equation}
Now $C^\bullet(A_1)$ is quasi-isomorphic to $\cH[-1]$ (as a direct computation of the cohomology of $C^\bullet(A_1)$ shows), and 
 the co-simplicial module $A_k$ is obtained from $A_1$ by composition with the functor $\wedge^k$. These two observations imply that $A_k$ is quasi-isomorphic to $\rL\!\wedge^k (\cH[-1])$, where $\rL\wedge^k$ is the left derived functor of $\wedge^k$ of Dold and Puppe \cite{DoldPuppe61}. By a theorem of Quillen \cite[Prop.~I.4.3.2.1(i)]{Illusie71}
we have a quasi-isomorphism
\[
	C^\bullet(A_k) \cong \rL\!\wedge^k (\cH[-1]) \cong (\Sym^k \cH)[-k],
\]
and together with (\ref{eq:double-simplicial-decomposed}) the theorem follows.
\end{proof}

\subsection{Cohomology of $\cM_1$ with $\bZ[\frac{1}{2}]$-coefficients}

Let $\pi\colon \cE \to \cM_{1,1}$ be the universal elliptic curve. By 
Theorem \ref{thm:stacks-summary} the stack $\cM_{1}$ coincides with the relative classifying space $\rB(\cE/\cM_{1,1})$. It follows that the homotopy type of $\cM_{1}^\an$ coincides with the homotopy type of (the geometric realization of) the simplicial space
\[
\xymatrix{
&\cdots
\ar@<1.5ex>[r]
\ar@<.5ex>[r]
\ar@<-.5ex>[r]
\ar@<-1.5ex>[r]
&
\cE^\an \times_{\cM_1^\an} \cE^\an
\ar@<1ex>[r]
\ar@<0ex>[r]
\ar@<-1ex>[r]
&
\cE^\an
\ar@<.5ex>[r]
\ar@<-.5ex>[r]
&
\cM_{1,1}^\an
}
\]
We have $\cM_{1,1}^\an \simeq \rB \SL_2 \bZ$, and $\cE^\an \to \cM_{1,1}^\an$ is the torus over $\rB \SL_2 \bZ$ corresponding to the standard representation of $\SL_2 \bZ$ on $\bZ^2$. Theorem \ref{thm:torus-split} with $\Lambda =\bZ^2$ and $G=\SL_2 \bZ$ now implies the following corollary. 

\begin{corollary}[Theorem \ref{thm:one} in the introduction] The cohomology of $\cM_{1}^\an$ satisfies
\[
\rH^n( \cM_{1}^\an\!,\, \bZ[{\textstyle\frac{1}{2}}] ) = \bigoplus_{p+2k=n} \rH^p(\SL_2 \bZ,\, \Sym^k  \bZ[\textstyle\frac{1}{2}]^2  )
\]
for all $n$. \qed
\end{corollary}

\begin{remark}The same arguments will give a similar description of the cohomology of the relative classifying space of the universal abelian variety over $\cA_g$, with coefficients 
in $\bZ[1/(2g)!]$.
\end{remark}

\section{Cohomology in low degree}\label{sec:low-degree}

We now turn to cohomology with $\bZ$-coefficients. Combining computations of Cohen, Callegaro, and Salvetti \cite{CallegaroCohenSalvetti13} and of
Furusawa, Tezuka, and Yagita \cite{Furusawa88} we will determine the integral cohomology groups of $\cM_{1}^\an$ in low degrees.

\subsection{Summary of Cohen-Callegaro-Salvetti}\label{subsec:summary-CCS}

Let $G=\SL_2 \bZ$, and $M_k$ the $G$-module $\Sym^k (\bZ^2)$. Callegaro, Cohen and Salvetti \cite[3.7, 3.8]{CallegaroCohenSalvetti13} have computed the groups $\rH^p( G, M_k) $ up to isomorphism for all $p$ and $k$. For low values of $k$ one has:
\[
\renewcommand{\arraystretch}{1.4}
\begin{tabular}{ c | c  c  c  c  }
$p$ & $p=0$ & $p=1$ & $p=2,4,6,\ldots$ & $p=3,5,7,\ldots$  \\
  \hline                        
$\rH^p(G,\, \bZ)$ &
	$\bZ$ &
	$0$ & 
	$\bZ/4\oplus \bZ/3$ & 
	$0$ \\
$\rH^p(G,\, M_1)$ &
	$0$ &
	$0$ & 
	$\bZ/2$ & 
	$0$ \\
$\rH^p(G,\, M_2)$ &
	$0$ &
	$\bZ\oplus \bZ/2 $ & 
	$0$ & 
	$(\bZ/2)^2$ \\
$\rH^p(G,\, M_3)$ &
	$0$ &
	$\bZ/2$ & 
	$\bZ/2$ & 
	$\bZ/2$  \\
$\rH^p(G,\, M_4)$ &
	$0$ &
	$\bZ\oplus \bZ/2 \oplus \bZ/3$ & 
	$\bZ/4$ & 
	$(\bZ/2)^2\oplus \bZ/3$ 
\end{tabular}
\]
For $p>1$  cupping with a generator $\gamma$ of $\rH^2(\SL_2 \bZ,\, \bZ ) \cong \bZ/12\bZ$ defines an isomorphism $\rH^p(G, M_i) \to \rH^{p+2}(G, M_i)$, making the table $2$-periodic in the $p$-direction.

\subsection{Summary of Furusawa-Tezuka-Yagita}

The ring $\rH^\bullet(\cM_1^\an, \bZ/2\bZ)$ has been computed by Furusawa, Tezuka and Yagita \cite[Thm.~4.11]{Furusawa88}. (Or rather: they compute $\rH^\bullet(\rB \Diff_+ T,\,\bZ/2\bZ)$, but see Remark \ref{rmk:teichmuller}).
They give a presentation with 10 generators, and many relations, see {\it loc.~cit.}. In low degree, the dimensions of the dagger-part are as follows:
\[
 \renewcommand{\arraystretch}{1.4}
\begin{tabular}{ c || c | c | c | c | c | c | c | c | c }
$n$ & $0$ & $1$ & $2$ & $3$ & $4$ & $5$ & $6$ & $7$ & $8$ \\
  \hline                        
$\dim_{\bF_2} \rH^n(\cM_1^\an,\, \bZ/2\bZ)^\dagger$ &
	$0$ & $0$ & $0$ & $1$ & $2$ & $3$ & $4$ & $5$ & $6$\\
\end{tabular}
\]
Moreover, we have $\dim_{\bF_2} \rH^n(\cM_{1,1}^\an,\,\bZ/2\bZ)=1$ for all $n$.

\subsection{Cohomology of $\cM_1$ in low degree}

\begin{theorem} \label{thm:low-degree-repeated}$\rH^n(\cM_{1}^\an,\,\bZ)^\dagger=0$ for $n\leq 3$. Moreover
\begin{enumerate}
\item $\rH^4(\cM_{1}^\an,\,\bZ)^\dagger \cong \bZ/2\bZ$
\item $\rH^5(\cM_{1}^\an,\,\bZ)^\dagger \cong \bZ\oplus\bZ/2\bZ$
\item $\rH^6(\cM_{1}^\an,\,\bZ)^\dagger \cong \bZ/2\bZ$
\item $\rH^7(\cM_{1}^\an,\,\bZ)^\dagger \cong (\bZ/2\bZ)^3$
\item $\rH^8(\cM_{1}^\an,\,\bZ)^\dagger \cong (\bZ/2\bZ)^2$
\item $\rH^9(\cM_{1}^\an,\,\bZ)^\dagger \cong \bZ\oplus (\bZ/2\bZ)^4\oplus \bZ/3\bZ$.
\end{enumerate}
\end{theorem}

This is a  more extended version of Theorem \ref{thm:low-degree} in the introduction. Note that the rank one parts in $\rH^5$ and $\rH^9$ are related to the Eisenstein series $E_4$ and $E_6$.

\begin{proof}[Proof of Theorem \ref{thm:low-degree-repeated}]
Consider the Leray spectral sequence for $J\colon \cM_1 \to \cM_{1,1}$ with
\[
	E_2^{p,q} = \begin{cases}
		\rH^p(\SL_2 \bZ,\, \Sym^k (\bZ^2)) & (q=2k) \\
		0 & (\text{$q$ odd}),
	\end{cases}
\]
converging to the cohomology of $\cM_1$. Using the `forgetful' section $\cM_{1,1} \to \cM_1$ one can split the spectral sequence as
\[
	E = E^\dagger \oplus E'
\]
with $E'$ concentrated in degrees $q=0$, and $E^\dagger$ concentrated in degrees $q>0$.  One has that $E^\dagger$ converges to 
$\rH^\bullet(\cM_1^\an, \bZ)^\dagger$, and $E'$ degenerates at $E'_2$ and gives the cohomology $\rH^\bullet(\cM_{1,1}^\an,\bZ)$.

The relevant part of $E^\dagger_2$ is as follows (using the table of \ref{subsec:summary-CCS}):
\[
\renewcommand{\arraystretch}{1.4}
\begin{tabular}{ c |   c   c   c c c  c c }
$q=8$ &  $\bZ\oplus\bZ/2\oplus\bZ/3$ & &  &  & & &  \\
$q=7$ &  $0$ & $0$ &  &  & & &  \\
$q=6$ &
	$\bZ/2$ & 
	$\bZ/2$ & 
	$\bZ/2$ & 
	 & & & 
	  \\
$q=5$  & $0$ & $0$ & $0$ & $0$ & \\
$q=4$  &
	$\bZ\oplus \bZ/2 $ & 
	$0$ & 
	$(\bZ/2)^2$ & 
	$0$ & $(\bZ/2)^2$ & & 
	 \\
$q=3$ & $0$ & $0$ & $0$ & $0$ & $0$ & $0$ & \\
$q=2$  &
	$0$ & 
	$\bZ/2$ & 
	$0$ & 
	$\bZ/2$ & 
	$0$ & 
	$\bZ/2$ & $0$ \\
	\hline
$E_2^{\dagger,p,q}$	 & $p=1$ & $p=2$ & $p=3$ & $p=4$ & $p=5$ & $p=6$ & $p=7$  \\

\end{tabular}
\]
Using the relation
\[
	\dim_{\bF_2} \rH^n(\cM_1^\an,\bZ/2\bZ)^\dagger = 
	\dim_{\bF_2} \rH^n(\cM_1^\an,\bZ)^\dagger\otimes \bF_2 +
	\dim_{\bF_2} \rH^{n+1}(\cM_1^\an,\bZ)^\dagger[2]
\]
one now verifies directly, that $E^\dagger$ must satifies
\begin{enumerate}
\item the differentials originating from an $E_k^{\dagger,p,q}$ with $p+q<10$  vanish
\item the filtration on $\rH^{n}(\cM_1^\an,\bZ)^\dagger$ with $n<10$ split,
\end{enumerate}
for otherwise $\rH^{p+q}(\cM_{1}^\an,\bF_2)^\dagger$ would be smaller than the group obtained by Furusawa, Tezuka and Yagita.
One finds for $n<10$ that $\rH^{n} = \oplus_{p+q=n} E_2^{p,q}$, which proves the theorem.
\end{proof}

When one tries to extend this strategy further, one runs into the problem of higher $2$-power torsion. For example, there is a copy of $\bZ/4\bZ$ in $E_2^{2,8}$ and the above numerical considerations will not be enough to decide if the spectral sequence degenerates in degree $10$ and higher.

\section{Characteristic classes with coefficients in $\bQ$ }
\label{sec:Q-coefficients}

\begin{theorem}
\label{thm:smooth-vanishing-repeated}
Let $\gamma \in \rH^\bullet(\cM_{1}^\an,\bQ)^\dagger$. 
\begin{enumerate}
\item If $S$ is smooth and of finite type over $\bC$
then $c_\gamma(C/S)=0$ for all genus one curves $C/S$;
\item If $\gamma\neq 0$ then there exists a scheme $S$ of finite type over $\bC$ and a $C/S$ so that $c_\gamma(C/S)\neq 0$.
\end{enumerate}
\end{theorem}

We will see in the proof that the curve $C\to S$ in (2) can be taken to be representable by a scheme.

\begin{proof}
(1). Assume $S$ is smooth and of finite type. Let $C\to S$ be a curve of genus one and let $E\to S$ be its Jacobian. By Theorem \ref{thm:Raynaud}, the torsor $C$ is of finite order in $\rH^1(S_\et,E)$. It follows that 
%there is an $N$ so that $C$ comes from a torsor in $\rH^1(S_\et,E[N])$. In particular,
there is a Galois finite \'etale $T\to S$ so that the base change $C_T\to T$ has a section. Since
$\gamma$ lies in the dagger part, we have $c_\gamma(C_T/T)=0$. Since we work with divisible coefficients, the pull-back map
\[
	\rH^\bullet(S^\an,\bQ) \to
	\rH^\bullet_\et(T^\an, \bQ)
\]
is injective (and the image consists of the Galois-invariant classes), and therefore also $c_\gamma(C/S)$ vanishes.

(2) Let $\gamma$ be a nonzero class in $\rH^n(\cM_{1}^\an,\bQ)$ with $n=2k-3$. We will construct a $C/S$ as in the theorem. The construction only depends on $k$. It is inspired by Example \ref{ex:infinite-order} at the end of Section \ref{sec:curves}. The reader may want to read that example first.

Let $\cC$ be the universal curve over $\cM_{1,k-1}$. It is equipped with sections \[
	P_0,\ldots,P_{k-2} \in \cC(\cM_{1,k-1}).
\]
Taking $P_0$ as identity we obtain an elliptic curve $\cE=(\cC,P_0)$ over $\cM_{1,k-1}$. Let $X$ be the nodal curve obtained by pinching $\bP^1$
along $\{0,\infty\}$. Consider the elliptic curve $\cE\times (\bP^1)^{k-2}$ over $\cM_{1,k-1}\times (\bP^1)^{k-2}$. Pinching the pair of sections
\begin{eqnarray*}
	(s; & x_1, \ldots, x_{i-1}, 0 , x_{i+1}, \ldots,  x_{k-2}) \\
	(s + P_i; & x_1, \ldots, x_{i-1}, \infty , x_{i+1}, \ldots,  x_{k-2}) 
\end{eqnarray*}
for all sections $s$ of $\cE$ and $x_j$ of $\bP^1$ we obtain a genus one curve $C$ over $S:=\cM_{1,k-1} \times X^{k-2}$.

We claim that $c_\gamma(C/S) \in \rH^{2k-3}(S^\an ,\bQ ) $ is non-zero. We will use the following lemma.

\begin{lemma} \label{lemma:concrete-H2} Let $T$ be a torus. Let $C\to T^k\times (S^1)^k$ be the principal $T$-bundle which is trivial over the covering space $T^k\times \bR^k \to T^k\times (S^1)^k$, and such that for every $(t_1,\ldots,t_k) \in T^k$ the monodromy along $\{(t_1,\ldots, t_k)\} \times (S^1)^k$ is given by
\[
	\pi_1( (S^1)^k, 1 )  \overset{\rho}{\cong} \bZ^k \to T, \, (n_1,\ldots, n_k) \mapsto n_1 t_1 + \cdots + n_k t_k
\]
Let $f\colon T^k\times (S^1)^k \to \rB T$ be the corresponding classifying map. Then there is a commutative square
\[
\begin{tikzcd}
	\Sym^k \rH^1(T,\,\bZ) \arrow{d}{\sim} \arrow{r}{\sim} &
	\rH^{2k}(\rB T,\, \bZ ) \arrow{d}{f^\ast}  \\
	\rH^1(S_1,\bZ)^{\otimes k} \otimes \Sym^k \rH^1(T,\,\bZ) \arrow[tail]{r} &
	\rH^{2k}(T^k\times (S^1)^k,\,\bZ),
\end{tikzcd}
\]
where the bottom map is (up to sign) the natural inclusion in 
\[
	\rH^{2k}(  T^k\times (S^1)^k,\,\bZ ) = \wedge^{2k} \big(\,\rH^1(T,\,\bZ)^{\oplus k} \oplus \rH^1(S^1,\,\bZ)^{\oplus k} \big),
\]
and the left map is induced by the orientation $\rho$.
\end{lemma}

\begin{remark}If $k=1$ and $T=S^1$ then $C$ is the Heisenberg manifold, the quotient of the $3\times 3$ upper triangular real matrices by the subgroup of upper triangular integral matrices.
\end{remark}

\begin{proof}[Proof of Lemma \ref{lemma:concrete-H2}]
Using the cup product on $\rB T$ and the K\"unneth formula for $T^k\times (S^1)^k$ one reduces to the case $k=1$.

Let $\Lambda = \rH_1( T,\,\bZ)$. Then the $T$-bundles on a manifold $S$ are classified by $\rH^2(S,\Lambda)$. For $T=S^1$ this is the classification by Chern class of an $S^1$-bundle, or equivalently, of a complex line bundle. We have $\rH^2( \rB T, \Lambda ) = \Hom( \Lambda, \Lambda )$, and the universal bundle on $\rB T$ corresponds to $\id_\Lambda$.

Now let $S=T\times (S^1)$. Then we have
\[
	\rH^2(S,\Lambda) = \Hom( \wedge^2 (\Lambda \oplus \bZ),\, \Lambda ).
\]
Let $C\to S$ be the $T$-bundle described in the lemma.  This bundle corresponds (up to sign) to the class
\[
	\wedge^2 \big( \Lambda \oplus \bZ \big) \to \Lambda,\,
	( x , n ) \wedge ( y, m ) \mapsto ny-mx
\]
in $\rH^2(S,\,\Lambda)$. Hence for a $\lambda \in \Hom(\Lambda,\bZ)=\rH^2(\rB T,\bZ)$ we have
\[
	f^\ast \lambda =
	\big[ ( x , n ) \wedge ( y, m ) \mapsto n\lambda(y)-m\lambda(x) \big]
	\in \rH^2( S, \bZ ),
\]
and one verifies directly that the square in the lemma commutes (up to sign).
\end{proof}

We now continue with the proof of Theorem \ref{thm:smooth-vanishing-repeated}.
We prove that  $c_\gamma(C/S) \in \rH^{2k-3}(S^\an ,\bQ ) $ is non-zero. Let 
$f\colon S\to \cM_1$ be the map corresponding to the genus one curve $C\to S$. We have a commutative triangle
\[
\begin{tikzcd}
S \arrow{rr}{f} \arrow[swap]{rd}{j} & & \cM_1 \arrow{ld}{J} \\
& \cM_{1,1} & 
\end{tikzcd}
\]
The fibers of $j$ above a point of $\cM_{1,1}$ corresponding to an elliptic curve $E$ is $E^{k-2} \times X^{k-2}$ which is homotopy-equivalent with $E^{k-2} \times (S^1)^{k-2}$.
Since $\cM_{1,1}$ has cohomological dimension $1$ (for $\bQ$-coefficients), the Leray spectral sequence for $J$ and for $j$ both collapse at $E_2$. We have
\[
	\rH^{2k-3}( S^\an, \bQ ) = \rH^1( \cM_{1,1}^\an,  \rR^{2k-4} j_\ast \bQ  ) \oplus \rH^0( \cM_{1,1}^\an,  \rR^{2k-3} j_\ast \bQ  )
\]
and
\[
	\rH^{2k-3}(\cM_1^\an, \bQ ) = \rH^1( \cM_{1,1}^\an, \Sym^{k-2} \rR^1 \pi_\ast \bQ ).
\]
The lemma provides a short exact sequence
\[
	0 \to \Sym^{k-2} \rR^1 \pi_\ast \bQ \overset{f^\ast}{\longto} \rR^{2k-4} j_\ast \bQ \longto \cQ \to 0.
\]
This sequence splits as a sequence of $\SL_2 \bZ$-modules, hence the induced map 
\[
	\rH^1( \cM_{1,1}^\an,\, \Sym^{k-2} \rR^1 \pi_\ast \bQ ) \overset{f^\ast}{\longto} 
	\rH^1( \cM_{1,1}^\an,\, \rR^{2k-4} j_\ast \bQ  )
\]
is injective, and $c_\gamma(C/S) \neq 0$. Finally, adding a suitable level structure on $\cM_{1,1}$ one finds a finite \'etale cover $S'\to S$ with $S'$ a {scheme}, and with $c_\gamma(C'/S') \neq 0$.
\end{proof}

\begin{remark}
For $\gamma$ coming from cusp forms, the vanishing of $c_\gamma(C/S)$ in part (1) of the theorem also follows from weight considerations. Let $S$ be a smooth scheme of finite type over $\bC$. Then the mixed Hodge structure $\rH^n(S^\an,\bQ)$ has weights $\geq n$ by Deligne's theorem \cite{DeligneHodgeII}. Yet, if $n=2k-3$, then the Hodge structure on $S_k \oplus \bar{S}_k$ is pure of weight $k-1 < n$, hence the pull-back map $S_k \oplus \bar{S}_k \to \rH^n(S^\an,\bC)$ is necessarily the zero map, and $c_\gamma(C/S)=0$.

Note that $\cM_1$ itself does not satisfy Deligne's weight inequality, despite $\cM_1$ being smooth.   See Sun \cite{Sun12} for a weaker inequality (in the $\ell$-adic context) that also holds for smooth \emph{stacks}.
\end{remark}

\section{An explicit torsion example}\label{sec:explicit-torsion-classes}

As we have seen in the previous section, the rational cohomology classes of $\cM_1$ give rise to characteristic classes that are rather pathological. For torsion cohomology classes, there is hope that they give non-trivial invariants for some interesting curves $C\to S$. 
 
From the computations in section \ref{sec:low-degree} we see that the first example of a (dagger) torsion class occurs in degree three, where we have
\[
	\rH^3(\cM_1^\an,\bZ/2\bZ)^\dagger = \bZ/2\bZ.
\]
In this section we will construct directly a characteristic class $c$ that associates to any $p\colon C\to S$ with $S$ a scheme over $\bZ[\frac{1}{2}]$ a class $c(C/S) \in \rH^3(S_\et,\bZ/2\bZ)$, and verify that it coincides with the non-zero class in $\rH^3(\cM_1^\an,\bZ/2\bZ)^\dagger$. We will also give an example of a $C\to S$ with $S$ a smooth scheme over $\bC$ and $c(C/S)\neq 0$.

\subsection{A canonical $E[2]$-torsor}\label{subsec:canonical-torsor}
For every elliptic curve $E\to S$ with $S$ a scheme over $\Spec \bZ[\frac{1}{2}]$ we will construct a canonical $E[2]$-torsor $T(E/S)$ on $S_\et$. The torsor only depends on the structure of $E[4]\to S$.

Let $M$ be a free module of rank $2$ over $\bZ/4\bZ$. Let $M^\ast \subset M$ be the set of elements of exact order $4$. Let $M[2]$ be the submodule of $2$-torsion and $M[2]^\ast$ the subset of elements of exact order $2$. Multiplication by $2$ induces a map
\[
	\varphi \colon M^\ast / \langle -1 \rangle \to M[2]^\ast.
\]
This map is a $2:1$ cover. The group $M[2]$ acts on the fibers by addition in $M[4]$. (Note that an element $a\in M[2]^\ast$ acts as the identity on the fiber $\varphi^{-1}(a)$, and as an involution on the two other fibers). Consider the set $T$ of (unordered) partitions of $\varphi$ into a disjoint union of two $1:1$ covers. Then $T$ consists of $4$ elements which are permuted transitively by $M[2]$. The construction of $T$ is natural, and we see that $T$ is an $\Aut M$-invariant $M[2]$-torsor. 

After the choice of an identification $M \cong (\bZ/4\bZ)^2$, one computes that
\[
	\left(\begin{array}{cc}1 & 1 \\0 & 1\end{array}\right) \in \GL_2 (\bZ/4\bZ)
\]
permutes the elements of $T$ cyclically. In particular, the torsor $T$ has no fixed points and hence is non-trivial. (In fact, one has 
$\rH^1( \GL_2(\bZ/4\bZ),\, (\bZ/2\bZ)^2 ) = \bZ/2\bZ$, so that $T$ is, up to isomorphism, the unique non-trivial torsor).

Now let $S$ be an algebraic stack over $\bZ[\frac{1}{2}]$ and $E \to S$ an elliptic curve. Applying the above construction to $E[4]$ one obtains a \emph{canonical} $E[2]$-torsor $T(E/S)$ on $S_\et$.

\begin{proposition}\label{prop:torsor-non-trivial}
Let $\pi\colon \cE \to \cM_{1,1}$ be the universal elliptic curve. Then
\[
	\rH^1(\cM_{1,1}^\an,\, \cE[2] ) \cong \bZ/2\bZ
\]
and the non-zero element is the class of the torsor $T(\cE/\cM_{1,1})$.
\end{proposition}

\begin{proof}
We have $\rH^1(\cM_{1,1}^\an,\, \cE[2] )=\rH^1(\SL_2 \bZ,\,(\bZ/2\bZ)^2)$.
Now, for the first statement, see for example \S \ref{subsec:summary-CCS}. For the second statement, it suffices to observe that by the above we have that the monodromy action of
\[
	\left(\begin{array}{cc}1 & 1 \\0 & 1\end{array}\right) \in \SL_2 \bZ
\]
has no fixed points, so $T(\cE/\cM_{1,1})$ has no sections.
 \end{proof}

\subsection{The characteristic class $c$}
We now use the above torsor to construct for every algebraic stack $S$ over $\bZ[\frac{1}{2}]$ and for every curve of genus one $p\colon C\to S$ a class $c(C/S) \in \rH^3(S_\et,\,\bZ/2\bZ)$. 

Let $E\to S$ be the Jacobian of $p\colon C\to S$. Let $T(E/S)$ be the associated $E[2]$-torsor constructed above. Note that there is a canonical isomorphism
\[
	\rR^1 p_\ast (\bZ/2\bZ)  = E[2]
\]
on $S_\et$ (since $E[2]$ is canonically isomorphic to its $\bF_2$-dual). In particular, the torsor $T(E/S)$ defines a class
\[
	t \in \rH^1( S_\et,\, \rR^1 p_\ast \bZ/2\bZ ).
\]	
Now the $E_2$-page of the Leray spectral sequence for $p\colon C \to S$ provides a differential
\[
	d\colon \rH^1( S_\et,\, \rR^1 p_\ast \bZ/2\bZ ) \to \rH^3( S_\et,\,  \bZ/2\bZ ),
\]
and we define $c(C/S)=d(t) \in \rH^3( S_\et,\,  \bZ/2\bZ )$. 

\begin{proposition}$c(C/S)$ is compatible with base change along arbitrary maps $S'\to S$ of schemes over $\bZ[\frac{1}{2}]$. If $C\to S$ has a section, then $c(C/S)=0$.
\end{proposition}

\begin{proof}
It is clear from the construction that $c(C/S)$ is compatible with base change along any $S'\to S$. If $C\to S$ has a section, then the bottom row of the Leray spectral sequence for $p\colon C\to S$ splits off, and the map $d$ becomes the zero map, hence $c(C/S)=d(t)=0$.
\end{proof}

\begin{proposition}Let $\gamma$ the unique non-zero element of $\rH^3(\cM_{1}^\an,\bZ/2\bZ)^\dagger$. Then for every scheme  $S$ of finite type over $\bC$ and for every curve of genus one $C\to S$, one has $c(C/S) = c_\gamma(C/S)$ in 
$\rH^3(S_\et,\, \bZ/2\bZ) = \rH^3(S^\an,\, \bZ/2\bZ)$.
\end{proposition}

\begin{proof}
Let $f \colon \cM_{1,1} \to \cM_1$ be the map that forgets a point. This coincides with the universal genus one curve over $\cM_1$. Clearly the construction of the class $c$ carries over to this situation, and yields a class
\[
	c(\cM_{1,1}/\cM_1) \in \rH^3( \cM_1^\an,\, \bZ/2\bZ ).
\]
Since $c(C/S)=0$ whenever $C\to S$ admits a section, we see that
\[
	c(\cM_{1,1}/\cM_1) \in \rH^3( \cM_1^\an,\, \bZ/2\bZ )^\dagger \cong \bZ/2\bZ.
\]
It therefore suffices to prove that $c(\cM_{1,1}/\cM_1)$ is non-zero. 

Let $\cE \to \cM_1$ be the Jacobian of $\cM_{1,1} \to \cM_1$. The associated $\cE[2]$-torsor $T(\cE/\cM_1)$ defines a class
\[
	t \in \rH^1 (\cM_1^\an,\, \cE[2] ) = \rH^1( \SL_2 \bZ, (\bZ/2\bZ)^2 ) \cong \bZ/2\bZ.
\]
which is non-zero by Proposition \ref{prop:torsor-non-trivial}.

Now consider the Leray spectral sequence for $\cM_{1,1} \to \cM_1$:
\[
	E_2^{p,q} = \rH^p( \cM_{1}^\an,\, \rR^q p_\ast \bZ/2\bZ )  \Longrightarrow \rH^{p+q}( \cM_{1,1}^\an,\, \bZ/2\bZ ).
\]
We have $\rH^2( \cM_{1,1}^\an, \bZ/2\bZ ) = \bZ/2\bZ$. Since $\cE[2]$ has no global sections we have $E_2^{0,1} = 0$. By the computations in low degree in Theorem \ref{thm:low-degree} we have $E_2^{2,0}=\bZ/2\bZ$. This $E_2^{2,0}$ contributes to 
$\rH^2( \cM_{1,1}^\an, \, \bZ/2\bZ)$, so $E_2^{1,1}$ cannot contribute. But this implies that the differential $d\colon E_2^{1,1} \to E_2^{3,0}$ must be injective. In other words,
\begin{equation}\label{eq:E-2-diff}
	\rH^1( \cM_1^\an,\, \cE[2] ) \overset{d}{\longto} \rH^3( \cM_{1}^\an,\, \bZ/2\bZ )
\end{equation}
is injective, and $c(\cM_{1,1}/\cM_1)$, being the image of $t$, is non-zero.
\end{proof}

\begin{remark}
The proof also shows that the non-trivial class in $\rH^3(\cM_1^\an,\,\bZ/2\bZ)^\dagger$ is the image of the class of the torsor $T$ under
\[
	\rH^1(\SL_2 \bZ,\, (\bZ/2\bZ)^2 ) = \rH^1(\cM_{1,1}^\an,\, \rR^1 J_\ast \bZ/2\bZ ) \longto \rH^3(\cM_1^\an,\,\bZ/2\bZ ),
\]
where the last map comes from the Leray spectral sequence for $J\colon \cM_1 \to \cM_{1,1}$.
\end{remark}

\subsection{An explicit non-vanishing example}

We now give an example of a \emph{smooth} scheme $S$ over $\bC$ and a curve of genus one $C\to S$ such that $c(C/S)$ is non-zero in $\rH^3(S^\an,\,\bZ/2\bZ)$.

Let $E$ be an elliptic curve over $\bC$ and let $X$ be an irreducible proper smooth scheme over $\bC$ such that
\begin{enumerate}
\item $\rH^1(X,\cO_X)=0$,
\item $\pi_1(X,x) = \bZ/2\bZ$.
\end{enumerate}
Such $X$ exist, for example one can take $X$ to be an Enriques surface. It follows from these conditions that the Albanese variety of $X$ is trivial, and that the torsion subgroup of $\rH^2(X^\an,\bZ)$ is $\bZ/2\bZ$. Let $\rH^2(X^\an,\bZ/2\bZ)^t$ be the image of $\rH^2(X^\an,\bZ)_\tors$ in $\rH^2(X^\an,\bZ/2\bZ)$. This is a one-dimensional subspace. 

Let $X'\to X$ be the universal cover of $X$. This is a Galois cover of degree $2$, and we denote by $\sigma$ the involution of $X'$. Let $P\in E(\bC)$ be a point of order $2$, and let $C\to X$ be the quotient of $E\times X'$ by the involution $(e,x) \mapsto (e+P,\sigma x)$. Then $C\to X$ is an $E$-torsor. Since the Albanese variety of $X$ vanishes we have $E(X)=E(\bC)$. It follows that the map
\[
	\rH^1(X^\an,\, E[2]) \to \rH^1(X^\an,\,E)
\]
is injective and hence that the torsor $C\to X$ is non-trivial. Let $f\colon X\to \rB E$ be the classifying map corresponding to $C$.

\begin{lemma}\label{lemma:enriques-torsion-class}
The image of $f^\ast \colon \rH^2(\rB E,\,\bZ/2\bZ) \to \rH^2(X,\,\bZ/2\bZ)$ is the rank one submodule
$\rH^2(X,\,\bZ/2\bZ)^t$.
\end{lemma}

\begin{proof}
Write $E=\bC/\Lambda$ and consider the corresponding short exact sequence
\[
	0 \to \Lambda \to \cO_{X^\an} \to E \to 0
\]
on $X^\an$. Since $\rH^1(X,\,\cO_X)=0$ the map $\rH^1(X^\an,\, E) \to \rH^2(X^\an,\, \Lambda)$ is injective. Let $x \in \rH^2(X^\an,\,\Lambda)$ be the image of the class of $C$. Then $x$ is a non-zero $2$-torsion element. In particular, the map
\[
	\Hom(\Lambda,\bZ/2\bZ) \overset{x}{\longto} \rH^2(X^\an,\,\bZ/2\bZ)
\]
has image in $\rH^2(X^\an,\,\bZ/2\bZ)^t$ and is non-zero. Since this map is precisely the map $f^\ast$, the lemma follows.
\end{proof}

Now let $n\geq 2$. Then $\Gamma_1(2n)$ is torsion free and the modular curve $Y_1(2n)_\bC$ is a fine moduli scheme. Let $S=X\times Y_1(2n)$ and let $\cE\to S$ be the pull-back of the universal elliptic curve over $Y_1(2n)$. By the level $\Gamma_1(2n)$-structure, the elliptic curve $\cE\to S$ is equipped with a nontrivial section $\xi\colon S\to \cE[2]$. Let $\cC\to S$ be the $\cE$-torsor  trivialized by $X'\times Y_1(2n) \to S$, and on which monodromy acts by translation over the $2$-torsion section $\xi$. The fibers over $Y_1(2n)$ are precisely the curves $C\to X$ constructed above.

\begin{theorem}$c(\cC/S)\neq 0$.
\end{theorem}

\begin{proof}

Consider the diagram
\[
\begin{tikzcd}
 S = X\times Y_1(2n) \arrow{r}{f} \arrow[swap]{dr}{j} & \cM_{1,\Gamma_{1}(2n)} \arrow{d}{J'} \arrow{r} & \cM_1 \arrow{d}{J} \\
 & Y_1(2n) \arrow{r} & \cM_{1,1},
\end{tikzcd}
\]
where $\cM_{1,\Gamma_{1}(2n)}$ is defined as to make the square cartesian. 

As in Lemma \ref{lemma:enriques-torsion-class} we have a rank one subsheaf $(\rR^2 j_\ast\bZ/2\bZ)^t$ of
$\rR^2 j_\ast \bZ/2\bZ$, and an exact sequence 
\[
0 \longto \bZ/2\bZ \longto  \rR^2 J'_\ast\bZ/2\bZ  \longto (\rR^2 j_\ast\bZ/2\bZ)^t \longto 0
\]
of local systems on $Y_1(2n)$. Since the middle term is isomorphic with $\cE[2]$, which has a unique non-trivial $\Gamma_1(2n)$-submodule (generated by $\xi$) this short exact sequence is isomorphic with the short exact sequence
\[
	0 \longto \bZ/2\bZ \overset{\xi}{\longto}
	 \cE[2] \overset{-\wedge \xi}{\longto}
	 \bZ/2\bZ \longto 0.
\]
We claim that the pull-back map
\[
	\rH^1(\cM_{1,1},\, \rR^2 J_\ast \bZ/2\bZ)  \to \rH^1(Y_1(2n),\, \rR^2 j_\ast \bZ/2\bZ),
\]
or equivalently,  that the composite map
\[
	\rH^1(\cM_{1,1},\, \cE[2] ) \to
	\rH^1(Y_1(2n),\, \cE[2]) \overset{-\wedge \xi}{\longto}
	\rH^1(Y_1(2n),\,\bZ/2\bZ)
\]
is injective.

We have $\rR^2 J_\ast \bZ/2\bZ = \cE[2]$, hence
the group $\rH^1(\cM_{1,1},\, \rR^2 J_\ast \bZ/2\bZ)$ has order $2$, generated by the class of the canonical $\cE[2]$-torsor $T=T(\cE/\cM_{1,1})$. 
The image of $T$ in $\rH^1(\Gamma_1(2n),\, E[2])$ is non-trivial, since $T$ has no fixed points under the action of
\begin{equation}\label{eq:upper-triangular}
	\left(\begin{array}{cc}1 & 1 \\0 & 1\end{array}\right) \in \Gamma_1(2n).
\end{equation}
There is a unique non-trivial $\Gamma_1(2n)$-equivariant map $\beta\colon E[2]  \to \bZ/2\bZ$. The $E[2]$-torsor $T$ induces a $\bZ/2\bZ$-torsor $\beta_\ast T$. One computes that the action of the matrix (\ref{eq:upper-triangular}) on $\beta_\ast T$ is non-trivial, and hence the map
\[
	\beta\colon \rH^1( \Gamma_1(2n),\, E[2] ) \to \rH^1( \Gamma_1(2n),\,\bZ/2\bZ )
\]
maps the class of $T$ to a non-zero class. Since any map from $E[2]$ to a $\bZ/2\bZ$-module with trivial $\Gamma_1(2n)$-action factors over the map $\beta$, we conclude that also the map
\[
	\rH^1( \Gamma_1(2n),\, E[2] ) \to \rH^1( \Gamma_1(2n),\, \rH^2(X,\, \bZ/2\bZ) )
\]
induced by  $f^\ast$ maps $T$ to a non-zero class. This proves the claim.

Finally, comparing the Leray spectral sequences for $S\to Y_1(2n)$ and for $\cM_1 \to \cM_{1,1}$ we find a commutative diagram
\[
\begin{tikzcd}
	\rH^1(\cM_{1,1},\, \rR^2 J_\ast \bZ/2\bZ) \arrow[tail]{r} \arrow[tail]{d} & \rH^3(\cM_{1,1},\,\bZ/2\bZ) \arrow{d} \\
	\rH^1(Y_1(2n),\, \rR^2 j_\ast \bZ/2\bZ) \arrow[tail]{r} & \rH^3( S,\,\bZ/2\bZ).
\end{tikzcd}
\]
The class $c(\cC/S) \in \rH^3(S,\,\bZ/2\bZ)$ is the image of the unique non-trivial element of $\rH^1(\cM_{1,1},\, \rR^2 J_\ast \bZ/2\bZ)$, and therefore 
$c(\cC/S)\neq 0$.
\end{proof}

%\begin{proposition}Let $n\geq 4$. Let $\cE \to Y_1(n)_{\bC}$ be the universal elliptic curve with a point of order $n$. Let $X/\bC$ be the nodal curve obtained by pinching $\bP^1$ along $\{0,\infty\}$. Let $S := \cE \times X$ and $C \to S$ be the genus one curve constructed in the proof of Theorem \ref{thm:smooth-vanishing-repeated}. Then $c(C/S) \neq 0$ in $\rH^3(S,\,\bZ/2\bZ)$.
%\end{proposition}
%
%\begin{proof}
% Let $\cM_{1,\Gamma_1(n)}$ be the pullback of $\cM_1$ along $Y_1(n) \to \cM_{1,1}$. That is, $\cM_{1,\Gamma_1{(n)}}$ parametrizes pairs $(C,P)$ of a genus one curve $C$ and a section $P$ of order $n$ in the Jacobian of $C/S$. We have $\rH^2(Y_1(n),\bF_2) = 0$. Considering the Leray spectral sequence for the forgetful map $Y_1(n) \to \cM_{1,\Gamma_1(n)}$ we find with the same argument as above that the map
%\[
%	\rH^1( \cM_{1,\Gamma_1(n)}^\an,\, \cE[2] ) \overset{d}{\longto} \rH^3( \cM_{1,\Gamma_1(n)}^\an,\, \bZ/2\bZ )
%\]
%(analogous to (\ref{eq:E-2-diff})) is injective. Also, the $\cE[2]$-torsor $T$ is non-trivial over $Y_1(n)$ (and over $\cM_{1,\Gamma_1{(n)}}$) since it has no fixed points under the monodromy action of
%\[
%	\left(\begin{array}{cc}1 & 1 \\0 & 1\end{array}\right) \in \Gamma_1(n).
%\]	
%This implies that the pullback map
%\[
%	\rH^3(\cM_1^\an,\,\bZ/2\bZ)^\dagger \to \rH^3(\cM_{1,\Gamma_1(n)}^\an,\,\bZ/2\bZ)
%\]
%is injective. Bla: end of proof, refer to Lemma in $\bQ$-section. Do it with integral coefficients.
%\end{proof}
%

We end by proving that $c(C/K)=0$ if $K$ is a number field.

\begin{proposition}Let $C\to \Spec \bR$ be a curve of genus one. Then $c(C/\bR)=0$.
\end{proposition}

\begin{proof}
Let $E/\bR$ be the Jacobian of $C$. The action of complex conjugation on $\rH_1(E(\bC),\bZ)$ is, on a suitable basis, given by one of the following two matrices
\[
	\left(\begin{array}{cc}0 & 1 \\1 & 0\end{array}\right), \quad
	\left(\begin{array}{cc}1 & 0 \\0 & -1\end{array}\right).
\]
In the former case one has $\rH^1(\bR, E)=0$ and in the latter case $\rH^1(\bR, E )=\bZ/2\bZ$. So we may assume we are in the second case. We then have $E[2](\bR) \cong (\bZ/2\bZ)^2$, and $E(\bR)$ has two connected components.

The Leray spectral sequence for $C\to S:=\Spec \bR$ gives an exact sequence
\[
	\rH^1(S_\et, \bZ/2\bZ) \longto \rH^1(C_\et, \bZ/2\bZ ) \longto \rH^0(S_\et, E[2] ) \overset{\alpha}{\longto} \rH^2(S_\et, \bZ/2\bZ ).
\]
The middle map factors as 
\[
	\rH^1(C_\et, \bZ/2\bZ ) = \rH^1(C_\et, \mu_2) \longto
	(\Pic C)[2] \longto E[2](\bR) = \rH^0(S_\et, E[2] ).
\]
The first map is surjective, and the second an isomorphism, so we see that the composite is surjective and hence $\alpha=0$.

Recall that the cohomology ring $\rH^\bullet(\Gal_\bR,\,\bZ/2\bZ)$ is isomorhic to $\bF_2[x,y]/(x^2)$ with $\deg x=1$ and $\deg y=2$.  Consider the square
\[
\begin{tikzcd}
\rH^0(S_\et, E[2] ) \arrow{r}{\alpha} \arrow{d}{\cdot x} 
	& \rH^2(S_\et, \bZ/2\bZ ) \arrow{d}{\cdot x} \\
\rH^1(S_\et, E[2] ) \arrow{r}{d} & \rH^3(S_\et, \bZ/2\bZ ),
\end{tikzcd}
\]
where $d$ is the map $E_2^{1,1} \to E_2^{1,3}$ from the Leray spectral sequence. Since this spectral sequence is compatible with cup product, the square commutes. Since $E[2]\cong (\bZ/2\bZ)^2$ the left map is surjective, and since $\alpha=0$ it follows that the bottom map $d$ vanishes.  We see that the image of the canonical torsor under $d$ vanishes and hence $c(C/\bR)=0$.
\end{proof}

\begin{corollary}Let $K$ be a number field and $C/K$ a curve of genus one. Then $c(C/K)=0$.
\end{corollary}

\begin{proof}
By \cite[Thm~3.1.c]{Tate63}, the natural map
\[
	\rH^3(\Gal_K,\, \bZ/2\bZ ) \longto \!\! 
	\prod_{\sigma\colon K \injto \bR} \!\!\! \rH^3(\Gal_\bR,\, \bZ/2\bZ)
\]
is an isomorphism.
\end{proof}

\end{document}